\documentclass[12pt]{amsart}
\usepackage[left=1.5in, right=1.5in, top=1.5in, bottom=1.5in]{geometry}                % See geometry.pdf to learn the layout options. There are lots.
\newtheorem{thm}{Theorem}[section]
\newtheorem{cor}[thm]{Corollary}

\newtheorem{prop}[thm]{Proposition}
\theoremstyle{definition}

\newtheorem{rem}[thm]{Remark}
\usepackage{graphicx}
\usepackage{amssymb}
\usepackage{amsmath}
\usepackage{epstopdf}
\usepackage{hyperref}
\title[The equilibrium shape]{On the equilibrium shape of a crystal}
\author{Emanuel Indrei}
\address{Department of Mathematics\\
Purdue University\\
West Lafayette, Indiana \\
USA.}

\begin{document}
\setcounter{page}{1}
\pagenumbering{arabic}
\maketitle

\begin{abstract}
A solution is given to a long-standing open problem posed by Almgren. 
\end{abstract}

\section{Introduction}

According to thermodynamics, the equilibrium shape of a small drop of water or a small crystal minimizes the free energy under a mass constraint. The phenomenon was independently discovered by W. Gibbs in 1878 \cite{G} and P. Curie in 1885  \cite{Crist}. Assuming the gravitational effect is negligible, the energy minimization is the surface area minimization and the solution is the convex set 

$$
K=  \bigcap_{v \in \mathbb{S}^{n-1}} \{x \in \mathbb{R}^n: \langle x, v\rangle< f(v)\}
$$
called the Wulff shape, where $f$ is a surface tension, i.e. a convex positively 1-homogeneous 
$$f:\mathbb{R}^n\rightarrow [0,\infty)$$
\\
with $f(x)>0$ if $|x|>0$ \noindent \cite{flashes, hilt, liebm, lau, MR0012454, MR82697, MR493671, MR1116536, MR1130601, MR1170247}. If $f(v)=R |v|$, i.e. the surface tension is isotropic, $K=B_R$ is the solution of the classical isoperimetric problem. 

Two main ingredients define the free energy of a set of finite perimeter $E \subset \mathbb{R}^n$ with reduced boundary $\partial^* E$: the surface energy

$$
\mathcal{F}(E)=\int_{\partial^* E} f(\nu_E) d\mathcal{H}^{n-1};
$$
and, the potential energy 
$$
\mathcal{G}(E)=\int_E g(x)dx,
$$
where $g \ge 0$, $g(0)=0$.
The free energy is the sum:
$$
\mathcal{E}(E)=\mathcal{F}(E)+\mathcal{G}(E).     
$$

In a gravitational field, the equilibrium shape for liquids was studied by P.S. Laplace in the early 1800s \cite{cap}. If the surface tension is isotropic, uniqueness and convexity were obtained by Finn \cite{MR607991, MR816345}  and Wente \cite{MR607986}. If $n=2$, subject to a wetting condition, the anisotropic tension was investigated by Avron, Taylor, and Zia via quadrature \cite{MR732374}. The work was motivated by low temperature experiments on helium crystals in equilibrium with a superfluid \cite{e1, e2, e3, e4, e5}. Also, various phase transition experiments were conducted in \cite{tr9, tr} and shape variations were investigated in \cite{trk} by  L.D. Landau.  

McCann considered the equilibrium shape for convex potentials with a bounded zero set \cite{MR1641031}. The central result is that the equilibrium planar crystals are a finite union of disjoint convex sets and each non-trivial component minimizes the free energy uniquely among convex sets of the same mass.  

Moreover, he proved that if the Wulff shape and potential are symmetric under $x \rightarrow -x$, there is a unique convex minimizer. Therefore, this provides information on the following problem mentioned in his paper:\\

Even when the field is the (negative) gradient of a convex potential, the equilibrium crystal is not known to be connected, much less convex or unique, \cite[p. 700]{MR1641031}. \\
 
If $g$ is locally bounded, Theorem \ref{@'} yields a stability result for small mass in any dimension which is stronger than uniqueness. 

Assuming $g \in C^1$ is coercive ($g(x)\rightarrow \infty$ as $|x|\rightarrow \infty$), $f \in C^{2, \alpha}(\mathbb{R}^n\setminus\{0\})$, $\alpha \in (0,1)$ is $\lambda-$elliptic, Theorem 2 in Figalli and Maggi \cite{MR2807136} yields convexity if the mass is small (a recent result obtained by Figalli and Zhang \cite{pFZ} implies that when $f$ is crystalline, if the mass is small, minimizers are polyhedra, cf. Remark \ref{Qzn}). Hence, the theorems together imply the uniqueness and convexity of minimizers if the mass is small. 
Theorem \ref{85} asserts that either: uniqueness and convexity hold for all masses; or, there exists $\mathcal{M}>0$ such that for all $m \in (0, \mathcal{M}]$, $E_m$ is unique and convex and for $m>\mathcal{M}$ there exists $a<m$ such that either convexity or uniqueness fails for mass $a$; or, the optimal critical mass $\mathcal{M}$ is exposed via 
\begin{equation*} \label{Mph}
\liminf_{m \rightarrow \mathcal{M}^-} \frac{\mathcal{M}^{\frac{n-1}{n}}-m^{\frac{n-1}{n}}}{w_m(\epsilon)} \ge \frac{1}{\gamma},
\end{equation*}
\\
where $w_m(\epsilon)>0$ is a modulus for the energy (see Proposition \ref{K}), $\gamma>0$ and $\epsilon \le \epsilon_0$, where $\epsilon_0$ is small. In two dimensions, the regularity assumptions are superfluous and the parameter which encodes the phase change is identified via Theorem \ref{8p}. Supposing the sub-level sets $\{g < t\}$ convex $\&$ a uniqueness assumption in the class of convex sets (valid for convex $g$), the critical mass is completely identified in Theorem \ref{8qp}. 

A convexity theorem for all masses necessarily must have restrictions on the potential. The following problem historically is attributed to Almgren \cite{MR2807136, MR1641031}.\\

{\bf Problem:} If the potential $g$ is convex (or, more generally, if the sub-level sets $\{g < t\}$
are convex), are minimizers convex or, at least, connected? \cite[p. 146]{MR2807136}.\\

\noindent I first proved that a convexity assumption is in general not sufficient.

\begin{thm} \label{g}
There exists $g \ge 0$ convex such that $g(0)=0$ $\&$ such that if $m>0$, then there is no solution to
$$
\inf\{\mathcal{E}(E): |E|=m\}.
$$ 
\end{thm}

Nevertheless, subject to additional assumptions, if the sub-level sets $\{g < t\}$
are convex, the convexity is true:

\begin{thm} \label{n=2}
If $n=2$ and\\ 
(i) g is locally Lipschitz in $\{g<\infty\}$\\
(ii) g admits minimizers $E_m \subset B_{R(m)}$ with $R \in L_{loc}^\infty(\mathbb{R}^+)$ \\ 
(iii) the sub-level sets $\{g<t\}$ are convex\\
(iv) when $E \subset \{g<\infty\}$ is bounded convex, $0 \notin E$, $|E \cap \{g \neq 0\}|>0$, then 
$$
\int_E \nabla g(x) dx \neq 0,
$$
then $E_m$ is convex for all $m \in (0,|\{g<\infty\}|)$.
\end{thm}

Note that if $g$ is convex, (i) is true. In addition, if $g$ is coercive, (ii) holds. In certain configurations, one may prove existence for non-coercive potentials, e.g. the gravitational potential \cite{MR732374, MR361996, MR420406}. Hence, (ii) is a natural assumption to ensure that the example in Theorem \ref{g} is precluded. In particular, since coercivity excludes the gravitational potential, a natural assumption is one that includes it; fortunately, it is not difficult to prove that (iv) includes the gravitational potential and therefore unifies the theory.

The result is the first theorem for $|\{g<\infty\}|=\infty$ with the weakest assumption formulated in Almgren's problem: convexity of sub-level sets.

If $g \in C^{1,\alpha}$ is strictly convex coercive, $f \in C^{3,\alpha}(\mathbb{R}^n\setminus \{0\})$ is uniformly elliptic, connectedness of equilibrium shapes was obtained by De Philippis and Goldman for $m \in (0,\infty)$ \cite{D}; if $n=2$ the strict convexity is not needed. Therefore, their theorem combined with McCann's theorem yields uniqueness and convexity in $\mathbb{R}^2$ subject to the regularity assumptions. If $g$ is convex, radial, and coercive, convexity of the equilibrium shapes for general convex surface tensions and $m \in (0,\infty)$ is a corollary of the above theorem without the regularity assumptions (Corollary \ref{ra}). 
\begin{cor}
If $n=2$ and\\ 
(i) $g=g(|x|)$ is locally Lipschitz \\
(ii) $g$ is coercive \\ 
(iii) when $g>0$, $g$ is (strictly) increasing \\
then $E_m$ is convex for all $m \in (0,\infty)$.
\end{cor}

\begin{cor}[The radial convex potential] \label{ra}
If $n=2$ and $g=g(|x|)$ is coercive and convex, then $E_m$ is convex for all $m \in (0,\infty)$.
\end{cor}

Assuming $g$ to be the gravitational potential, existence $\&$ convexity was proved by Baer in higher dimension with a structural assumption on the surface tension \cite{R}, cf. \cite{MR2126076, MR2321891, MR2349864, MR2875653, MR872883, MR493670, MR543126, MR543127, MR872883}. Assuming smoothness on the surface tension, he proved uniqueness. If $m$ is small, Corollary  \ref{@'w09} yields uniqueness for more general gravitational potentials and without the additional regularity assumption. 
In two dimensions, the above theorem implies convexity when the surface tension admits minimizers (e.g. $f$ is admissible \cite{R}). 

\begin{cor} \label{2-gr}
If $n=2$, $\phi(0)= 0$, $\phi'>0$, and\\ 
(i) 
\begin{equation*}  
g(x)=\begin{cases}
 \phi(x_2) & \text{if } x_2\ge 0\\
\infty  & \text{if }x_2<0
\end{cases}
\end{equation*}
(ii) $\mathcal{F}$ satisfies assumptions for the existence of minimizers $E_m \subset B_{R(m)}$ with $R \in L_{loc}^\infty(\mathbb{R}^+)$, \\

\noindent then $E_m$ is convex for all $m \in (0,\infty)$.
\end{cor}

Moreover, with convexity of sub-level sets, Theorem \ref{mi} yields that minimizers are a finite union of convex sets with disjoint closures. Therefore, this contains the geometry in McCann's result when $g$ is assumed convex with bounded zero set. In my proofs of Theorems  \ref{n=2} $\&$ \ref{q} (and Corollary \ref{m+e}), Theorem \ref{mi} is an important technical tool.

The technique implies the following: if $A \subset \mathbb{R}^2$ is bounded, convex and 
$$\mathcal{F}(E_m)=\inf\{\mathcal{F}(E): E\subset A, |E|=m\},$$ with $|K_a|<m\le |A|$ $\&$ $|K_a|$ the measure of the largest Wulff shape in $A$,  then $E_m$ is convex. Interestingly, understanding the convexity in the case of the isotropic perimeter was the main objective in \cite{MR1669207} $\&$ the problem has appeared in a few subsequent papers: \cite[Theorem 3.32.]{MR1669207}, \cite[Theorem 11]{MR2178065}, \cite[(1.8)]{MR2436794}, and \cite[Remark 2.6.]{MR2468216}.
One feature is to investigate the problem within a wider framework also inclusive of the Cheeger constant of a fixed domain. 
Supposing the surface tension to be even $\&$ a special condition involving $A$ and $f$, convexity for an interval was proved when $n \ge 2$ in \cite[Theorem 6.5.]{MR2436794} (cf.  \cite{MR3719067}).
Similar problems are investigated in convergence of curvature flows \cite{MR2248685, MR3981988, MR4025327, MR2558422, MR2208291, MR1205983, MR1078266, MR1087347}.

\section{Main theorems}

\subsection{Proof of Theorem \ref{g}}

\begin{proof}
Define
\begin{equation*}  
g(x,y)=\begin{cases}
x^2(1-y)+x^2y^2 & \text{if } y\le0\\
\frac{x^2}{1+y}  & \text{if } y > 0.
\end{cases}
\end{equation*}
Note that 
\begin{equation*}  
D^2g(x,y)=\begin{cases}
\begin{pmatrix}
  2(1-y+y^2) & 4xy-2x \\
  4xy-2x & 2x^2
 \end{pmatrix}
& \text{if } y\le0\\
\begin{pmatrix}
  \frac{2}{1+y} & -\frac{2x}{(1+y)^2} \\
  -\frac{2x}{(1+y)^2} & \frac{2x^2}{(1+y)^3}
 \end{pmatrix}
 & \text{if } y > 0;\\
\end{cases}
\end{equation*}
let $\lambda_1, \lambda_2$ be the eigenvalues of $D^2g(x,y)$, 
\begin{equation*}  
\lambda_1(x,y)\lambda_2(x,y)=\det(D^2g(x,y))=\begin{cases}
12x^2y(1-y)
& \text{if } y\le0\\
0
 & \text{if } y > 0,
\end{cases}
\end{equation*}
\begin{equation*}  
\lambda_1(x,y)+\lambda_2(x,y)=\begin{cases}
2(1-y+y^2+x^2) 
& \text{if } y\le0\\
 \frac{2}{1+y}\Big(1 + \frac{x^2}{(1+y)^2}\Big)
 & \text{if } y > 0.\\
\end{cases}
\end{equation*}   
\begin{figure}[htbp]
\caption{$g$}
\label{f211}
\centering
\includegraphics[width=.6 \textwidth]{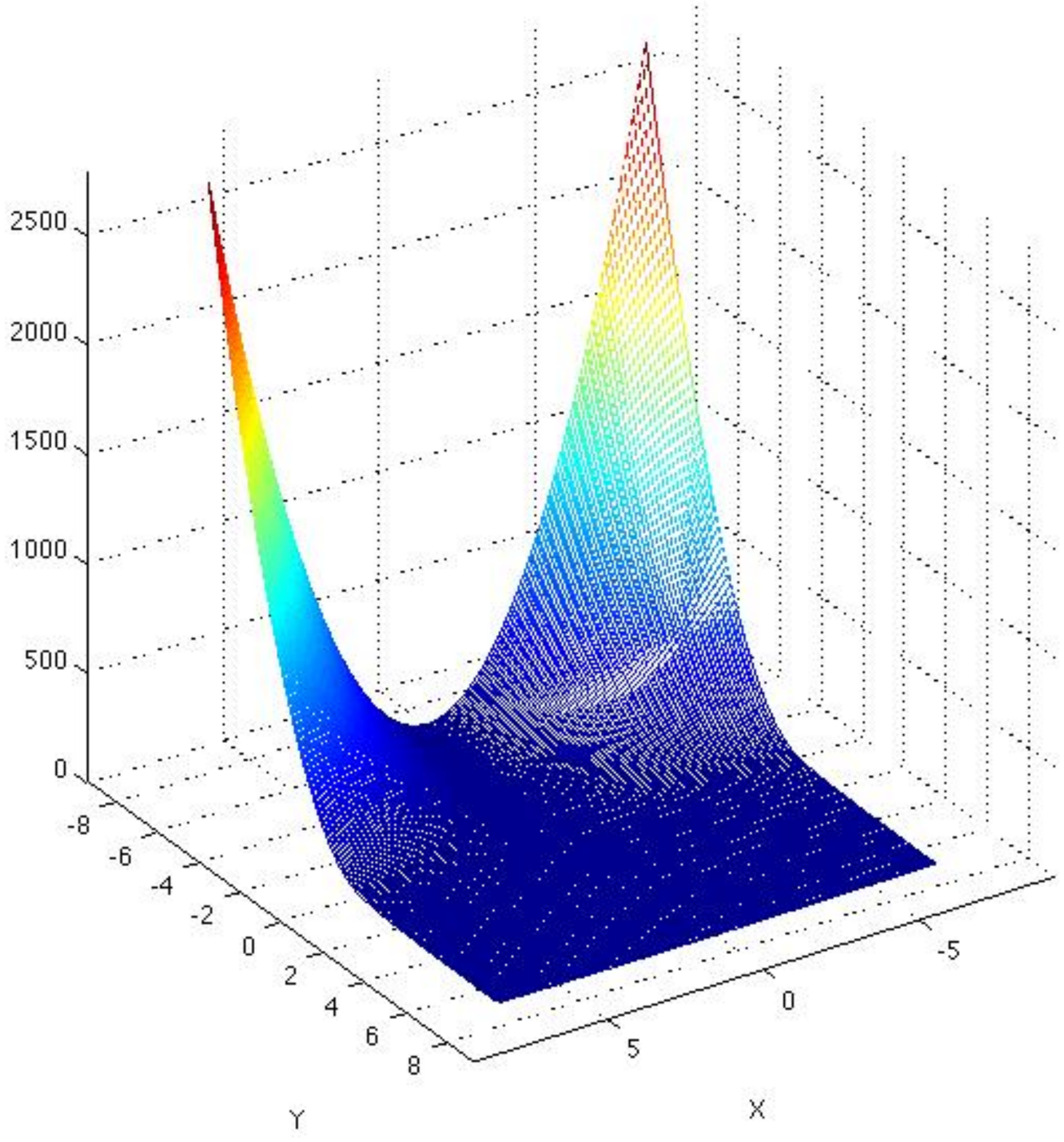}
\end{figure}

In particular, 

$$
\lambda_1(x,y), \lambda_2(x,y) \ge 0.
$$
\\
Hence, $D^2g(x,y)$ is a real non-negative Hermitian matrix and therefore can be diagonalized by a non-negative diagonal matrix $\Lambda(x,y)$ and a real orthogonal matrix $O(x,y)$ which has as columns real eigenvectors: 
$$
D^2g(x,y)=O(x,y)\Lambda(x,y) O(x,y)^T.
$$
\\
\noindent Assume $w \in \mathbb{R}^2 \setminus \{0\}$ and set $z=O(x,y)^T w$;
\begin{align*}
\langle D^2g(x,y) w, w \rangle&= \langle O(x,y)\Lambda(x,y) O(x,y)^T w, w \rangle\\
&=\langle \Lambda(x,y) O(x,y)^T w, O(x,y)^T w \rangle \\
&= \lambda_1(x,y)|z_1|^2+\lambda_2(x,y)|z_2|^2 \ge 0
\end{align*}
and this implies that $g$ is convex, see Figure \ref{f211}.
Set  $e_2=(0,1)$, $a>0$; the potential is non-increasing in the $y$-variable and strictly decreasing if $x\neq 0:$
$$
\partial_y g(x,y)=
\begin{cases}
x^2(-1+2y)
& \text{if } y\le0\\
-\frac{x^2}{(1+y)^2}
 & \text{if } y > 0;\\
\end{cases}
$$
in particular, if a minimizer $E_m$ exists,
$$
\int_{E_m+ae_2} g(x,y) dxdy < \int_{E_m} g(x,y) dxdy, 
$$
\begin{figure}[htbp]
\centering
\includegraphics[width=.6 \textwidth]{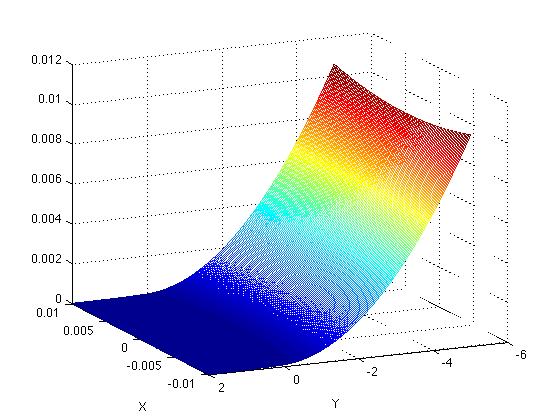}
\caption{$g^\epsilon$}
\label{a12}
\end{figure}

\noindent which yields
$$
\mathcal{E}(E_m+ae_2)< \mathcal{E}(E_m),
$$
a contradiction.
Observe that one may also concatenate another function at $y=0$: let $\epsilon>0$ $\&$ define 
$$
g^\epsilon(x,y)=
\begin{cases}
x^2(1-y)+\frac{\epsilon}{2}y^2
& \text{if } -\sqrt{\frac{\epsilon}{2}} \le x \le \sqrt{\frac{\epsilon}{2}}, y \le 0\\
\frac{x^2}{1+y}
 & \text{if } -\sqrt{\frac{\epsilon}{2}} \le x \le \sqrt{\frac{\epsilon}{2}}, y >0\\
\end{cases}
$$
$\&$ extend $g^\epsilon$ by a convex envelope, see Figure \ref{a12}. 
\end{proof}

\begin{rem}
Note that to apply sharp theorems on the extension of $g^\epsilon$, the corners of the domain can be smoothed and the extension then inherits up to $C_{loc}^{1,1}$ regularity \cite{MR3324933} (an application to PDEs is given in \cite{MR3933403}). 
\end{rem}

\subsection{A stability theorem}

The counterexample shows that in a general context, more assumptions are necessary to ensure existence. A well-known assumption is coercivity. Nevertheless, coercivity excludes the gravitational potential. In order to obtain convexity without relying on constrained settings for existence, {\bf g admits minimizers} is defined to mean any assumption generating a minimizer: one may e.g. obtain existence with $\mathcal{F}=\mathcal{H}^{n-1}$ and the gravitational potential utilizing Steiner symmetrization \cite{MR2178968, MR3055761} to avoid a sequence of sets escaping to infinity to prevent a compactness argument. Other situations specific to the potential or surface tension likewise may generate existence although not included under coercivity.

\begin{thm} \label{@'}
Suppose $g \in L_{loc}^\infty(\{g<\infty\})$ admits minimizers $E_m \subset B_{R}$ if $m$ is small. For all  $\epsilon>0$ there exists $c_\epsilon>0$, $a=a(\cdot, \epsilon)$,  $m_0>0$ such that
$$
\inf_{y>0} \frac{a(y,\epsilon)}{y^{\frac{n-1}{n}}} \ge c_\epsilon,
$$ 
$\&$ for all $E \subset B_R$, $|E|=|E_m|=m<m_0$, if
$$
|\mathcal{E}(E_m)-\mathcal{E}(E)| < a(m,\epsilon),
$$
there exists $x_0$ such that  
$$
\frac{|(E+x_0) \Delta E_m|}{|E_m|} < \epsilon.
$$
\end{thm}

\begin{proof}
Assume the theorem is false. Then there exists $\epsilon>0$ such that for all $c>0$, $a(\cdot, \epsilon)$ such that
$$
\inf_{y>0} \frac{a(y,\epsilon)}{y^{\frac{n-1}{n}}} \ge c;
$$
for all $m_0 \in (0,\infty)$ there exist $m<m_0$ and  minimizers $E_{m} \subset B_R$ \& sets $E_{m}' \subset B_R$ with  
$$|E_{m}|=|E_{m}'|=m$$ 
such that
$$
|\mathcal{E}(E_m)-\mathcal{E}(E_{m}')| < a(m,\epsilon)
$$ 
and

$$
\inf_{x_0 \in \mathbb{R}^n} \frac{|(E_{m}'+x_0) \Delta E_{m}|}{|E_{m}|} \ge \epsilon>0.
$$
Let $w_k \rightarrow 0^+$, $q$ a modulus of continuity ($q(0^+)=0$), and define $c_k=w_k q(\epsilon)$,
$a_k(y,\epsilon)=c_k y^{\frac{n-1}{n}}$; now, 
$$
\inf_{y>0} \frac{a_k(y,\epsilon)}{y^{\frac{n-1}{n}}} = c_k
$$
and selecting $m_0=\frac{1}{k}$ for $k \in \mathbb{N}$, there exist minimizers $E_{m_k}\subset B_R$ and sets $E_{m_k}' \subset B_R$, $|E_{m_k}|=|E_{m_k}'|=m_k<\frac{1}{k}$ such that
$$
|\mathcal{E}(E_{m_k})-\mathcal{E}(E_{m_k}')| < a_k(m_k,\epsilon),
$$ 
and
$$
\inf_{x_0 \in \mathbb{R}^n} \frac{|(E_{m_k}'+x_0) \Delta E_{m_k}|}{|E_{m_k}|} \ge \epsilon>0.
$$
Set $a_k=a_k(m_k,\epsilon)$, $\gamma_k=(\frac{|K|}{m_k})^{\frac{1}{n}}$,

$$|\gamma_k E_{m_k}|=|K|;$$  
since $m_k \rightarrow 0$, 

$$
\delta(\gamma_k E_{m_k}):=\frac{\mathcal{F}(\gamma_k E_{m_k})}{n|K|^{\frac{1}{n}}|\gamma_k E_{m_k}|^{\frac{n-1}{n}}} - 1 \rightarrow 0
$$

(via e.g. Corollary 2 in Figalli and Maggi \cite[p. 176]{MR2807136}). 

By the triangle inequality,

\begin{align*}
|\mathcal{F}(E_{m_k}')&-\mathcal{F}(E_{m_k})|\\
&=|[\mathcal{E}(E_{m_k}')-\mathcal{E}(E_{m_k})]+[\int_{E_{m_k}} g(x)dx-\int_{E_{m_k}'} g(x)dx]|\\
&\le |\mathcal{E}(E_{m_k}')-\mathcal{E}(E_{m_k})|+\int g(x)|\chi_{E_{m_k}'}-\chi_{E_{m_k}}| dx\\
&<a_k+\int_{E_{m_k}' \Delta E_{m_k}} g(x)dx.
\end{align*}

Multiplying both sides by $\gamma_k^{n-1}$, 

\begin{align*}
|\mathcal{F}(\gamma_kE_{m_k}')&-\mathcal{F}(\gamma_kE_{m_k})|\\
&<\gamma_k^{n-1}a_k+2\gamma_k^{n-1} m_k(\sup_{B_R \cap \{g<\infty\}} g) \\
&=|K|^{\frac{n-1}{n}}\frac{a_k}{m_k^{\frac{n-1}{n}}}+2|K|^{\frac{n-1}{n}}(\sup_{B_R \cap \{g<\infty\}} g) m_k^{1-\frac{n-1}{n}}
\end{align*}

and since $a_k=a_k(m_k, \epsilon)=c_k m_k^{\frac{n-1}{n}}=w_k q(\epsilon)m_k^{\frac{n-1}{n}}$,

$$
\frac{a_k}{m_k^{\frac{n-1}{n}}}=w_k q(\epsilon) \rightarrow 0
$$

$$
 |\mathcal{F}(\gamma_kE_{m_k}')-\mathcal{F}(\gamma_kE_{m_k})| \rightarrow 0
$$

\begin{align*}
\delta(\gamma_k E_{m_k}') & \le |\delta(\gamma_k E_{m_k}')-\delta(\gamma_k E_{m_k})|+\delta(\gamma_k E_{m_k})\\
&= \frac{1}{n|K|}|\mathcal{F}(\gamma_kE_{m_k}')-\mathcal{F}(\gamma_kE_{m_k})|+\delta(\gamma_k E_{m_k})\\
&\hskip .15in \rightarrow 0.
\end{align*}

Hence, there exist $x_k, x_k' \in \mathbb{R}^n$ such that

$$
\frac{|(\gamma_kE_{m_k}+x_k) \Delta K|}{|\gamma_k E_{m_k}|} \rightarrow 0,
$$
\&
$$
\frac{|(\gamma_kE_{m_k}'+x_k') \Delta K|}{|\gamma_k E_{m_k}'|} \rightarrow 0
$$

via a compactness argument \footnote{or, the stability of the anisotropic isoperimetric inequality \cite{MR2672283}}:
if this is not true, then up to a subsequence

$$
\inf_x \frac{|(\gamma_kE_{m_k}+x) \Delta K|}{|\gamma_k E_{m_k}|} \ge a>0;
$$
let  $E_k:=\gamma_k E_{m_k}$ $\&$ observe

$$
\sup_k \mathcal{F}(E_k)<\infty,
$$
hence up to a subsequence, $E_k \rightarrow E$ in $L_{loc}^1$, $|E|=|E_k|=|K|$,
$$
\mathcal{F}(E) \le \liminf_k \mathcal{F}(E_k)=\mathcal{F}(K) \le \mathcal{F}(E)
$$ 
via the anisotropic isoperimetric inequality. Therefore, there exists $x \in \mathbb{R}^n$ such that
$$
0<a \le \frac{|(\gamma_kE_{m_k}+x) \Delta K|}{|\gamma_k E_{m_k}|} \rightarrow \frac{|(E+x) \Delta K|}{|K|} =0,
$$
a contradiction. Similarly, a symmetric argument implies
$$
\frac{|(\gamma_kE_{m_k}'+x_k') \Delta K|}{|\gamma_k E_{m_k}'|} \rightarrow 0
$$
and this yields $k \in \mathbb{N}$ (via the triangle inequality in $L^1$ applied to characteristic functions) such that 
$$
\frac{|(E_{m_k}'+\frac{(x_k'-x_k)}{\gamma_k}) \Delta E_{m_k}|}{|E_{m_k}|}<\epsilon,
$$
a contradiction to
$$
\inf_{x_0 \in \mathbb{R}^n} \frac{|(E_{m_k}'+x_0) \Delta E_{m_k}|}{|E_{m_k}|} \ge \epsilon>0.
$$
\end{proof}

\begin{rem}
The result may also be extended to $g \in L_{loc}^1(\{g<\infty\})$ subject to some assumptions. In the extension, Lebesgue's differentiation theorem is utilized.
\end{rem}

\begin{cor} \label{@'w}
Suppose $g \in L_{loc}^\infty$ is coercive. For all  $\epsilon>0$ there exists $c_\epsilon>0$, $a=a(\cdot, \epsilon)$,  $m_0>0$ such that
$$
\inf_{y>0} \frac{a(y,\epsilon)}{y^{\frac{n-1}{n}}} \ge c_\epsilon,
$$ 
$\&$ for a minimizer $E_m \subset B_R$, $E \subset B_R$, $|E|=|E_m|=m<m_0$, if
$$
|\mathcal{E}(E_m)-\mathcal{E}(E)| < a(m,\epsilon),
$$
there exists $x_0$ such that  
$$
\frac{|(E+x_0) \Delta E_m|}{|E_m|} < \epsilon.
$$
Therefore, if $m<m_0$, $E_m$ is (mod translations and sets of measure zero) unique.
\end{cor}

\begin{rem} 
Suppose $g \in L_{loc}^\infty$ $\&$ $n=2$, then there exists $m_0>0$ such that for $m<m_0$, $E_m$ is unique and convex (via combining Corollary \ref{@'w} with Theorem 1 in Figalli and Maggi).
\end{rem}

\begin{rem} \label{Qzn}
Assuming $g$ is coercive, convex (strictly for $n >2$), and additional regularity, 
De Philippis and Goldman \cite{D} proved minimizers are connected and therefore if $n=2$, convex $\&$ unique via McCann \cite{MR1641031}. In the theorem, I obtain uniqueness for $m$ small without regularity, a convexity assumption, and displacement interpolation introduced in  \cite{Mcq, Mc, MR1641031}; and for $n\ge2$. This also implies a new proof of convexity when $n=2$ in the convex case: De Philippis and Goldman in particular prove that there exist convex minimizers for $m \in \mathbb{R}^+$ \cite[Corollary 1.2., Remark 1.3.]{D} and if $m<m_0$, Theorem \ref{@'} implies, up to translations and sets of measure zero, that no others exist. 

The small mass convexity for $n=2$ has already been proven in Figalli and Maggi's Theorem 1. The result in the corollary improves the convexity with the uniqueness. To the best of my knowledge, this is the first uniqueness and convexity result for small mass with merely $g \in L_{loc}^\infty(\mathbb{R}^2)$ and general convex surface tension. If $f$ is crystalline, Figalli and Zhang obtained the geometry of minimizers in higher dimension: for sufficiently small mass minimizers are polyhedra \cite{pFZ}. Assuming a stability condition, Corollary \ref{m_s} implies convexity for small mass with $g \in L_{loc}^\infty(\mathbb{R}^n).$
\end{rem}

Also, new results on the gravitational potential are obtained via the theorem.

\begin{cor} \label{@'w09}
Suppose $0 \le \phi \in L_{loc}^\infty([0,\infty))$ and\\ 
(i) 
\begin{equation*}  
g(x)=\begin{cases}
 \phi(x_n) & \text{if } x_n\ge 0\\
\infty  & \text{if }x_n<0
\end{cases}
\end{equation*}
(ii) $\mathcal{F}$ satisfies assumptions for the existence of minimizers $E_m \subset B_{R(m)}$ with $R \in L_{loc}^\infty(\mathbb{R}^+)$,\\ 

\noindent then if $m>0$ is small, $E_m$ is unique. Moreover, for all  $\epsilon>0$ there exists $c_\epsilon>0$, $a=a(\cdot, \epsilon)$,  $m_0>0$ such that
$$
\inf_{y>0} \frac{a(y,\epsilon)}{y^{\frac{n-1}{n}}} \ge c_\epsilon,
$$ 
$\&$ for $E \subset B_R$, $|E|=|E_m|=m<m_0$, if
$$
|\mathcal{E}(E_m)-\mathcal{E}(E)| < a(m,\epsilon),
$$
there exists $x_0$ such that  
$$
\frac{|(E+x_0) \Delta E_m|}{|E_m|} < \epsilon.
$$
\end{cor}

\begin{cor} \label{q,}
Suppose $\alpha>0$,\\ 
(i) 
\begin{equation*}  
g(x)=\begin{cases}
 \alpha x_n & \text{if } x_n\ge 0\\
\infty  & \text{if }x_n<0
\end{cases}
\end{equation*}
(ii) $f$ is admissible \cite{R},\\ 

\noindent then if $m>0$ is small, $E_m$ is unique. Moreover, for all  $\epsilon>0$ there exists $c_\epsilon>0$, $a=a(\cdot, \epsilon)$,  $m_0>0$ such that
$$
\inf_{y>0} \frac{a(y,\epsilon)}{y^{\frac{n-1}{n}}} \ge c_\epsilon,
$$ 
$\&$ for $E \subset B_R$, $|E|=|E_m|=m<m_0$, if
$$
|\mathcal{E}(E_m)-\mathcal{E}(E)| < a(m,\epsilon),
$$
there exists $x_0$ such that  
$$
\frac{|(E+x_0) \Delta E_m|}{|E_m|} < \epsilon.
$$
\end{cor}

\begin{rem}
Assuming that $f$ is smooth, Baer proved uniqueness \cite[Theorem 3.12.]{R}. Corollary \ref{q,} together with \cite[Theorem 3.10.]{R} yields a unique convex minimizer when $m$ is small under the assumptions in Corollary \ref{q,} (in particular, without additional regularity). 
\end{rem}

\subsection{Geometry of $E_m$}

Next, the initially stated problem is revisited to address results for $m \in (0,\infty)$.\\

{\bf Problem:} If the potential $g$ is convex (or, more generally, if the sub-level sets $\{g < t\}$
are convex), are minimizers convex or, at least, connected?\\

\begin{thm} \label{mi}
If $n=2$, the sub-level sets $\{g < t\}$ are convex, $g$ is locally Lipschitz, and $g$ admits minimizers $E_m \subset B_{R(m)}$, then 
$$E_{m}=\cup_{i=1}^N A_{i},$$
where $A_{i}$ are convex and have disjoint closures, $N<\infty$.
\end{thm}

\begin{proof}

The first variation formula for $\mathcal{F}$ implies that the anisotropic mean
curvature of $\partial E_{m}$ is non-negative and in two dimensions this is sufficient for the convexity of 
every connected component of $E_{m}$ (via e.g. Proposition 1 in Figalli and Maggi \cite{MR2807136}).

Therefore, 
$$E_{m}=\cup_{i=1}^\infty A_{i},$$
where $A_{i}$ are disjoint, convex. Theorem 4.4 in Giusti \cite{Giust} yields -- up to sets of measure zero --

$$
\partial E_m = \overline{\partial^* E_m},
$$
$\&$ density estimates imply

$$
\mathcal{H}^1(\partial E_m \setminus \partial^* E_m)=0;
$$

hence, 

$$\mathcal{H}^1(\partial \overline{E_m} \setminus \partial \text{int} E_m)=0.$$
\\
This implies that $\text{int} E_m$ is also a minimizer. If $A_j$ and $A_l$ ($j \neq l$) have non-disjoint closures, let $x \in \partial^* E_m\cap \overline{A_j} \cap \overline{A_l} $ and note that thanks to the regularity $x$ has density $1$, but since 
$x \notin \text{int} E_m$, the density estimate (e.g. Proposition A.6 in \cite{MR1641031} applied with $\text{int} E_m$) implies 
$$
|B_r(x) \setminus \text{int} E_m| \ge a r^2>0
$$
when $r>0$ is small, a contradiction. Suppose $x \in (\partial E_m \setminus \partial^* E_m)\cap \overline{A_j} \cap \overline{A_l}$; thanks to 
$
\mathcal{H}^1(\partial E_m \setminus \partial^* E_m)=0
$
($\&$ convexity of $A_{i}$)
$\partial A_j \cap \partial A_l=\{x\}.$ Let $E=E_m \cup \text{conv}(A_j \cup A _l)$: firstly observe that $|E|>m$, and since the closure of $A_j \cup A _l$ is connected, 
$$
\mathcal{F}(\text{conv}(A_j \cup A _l)) \le \mathcal{F}(A_j \cup A _l)
$$
 (via e.g. Corollary 2.8 in \cite{MR1641031}); in addition, the surface energy satisfies the inclusion-exclusion estimate
 $$
 \mathcal{F}(A \cup B)+\mathcal{F}(A \cap B) \le \mathcal{F}(A)+\mathcal{F}(B)
 $$
 
 \begin{figure}[htbp] 
\centering
\label{nwing}
\includegraphics[width=.6 \textwidth]{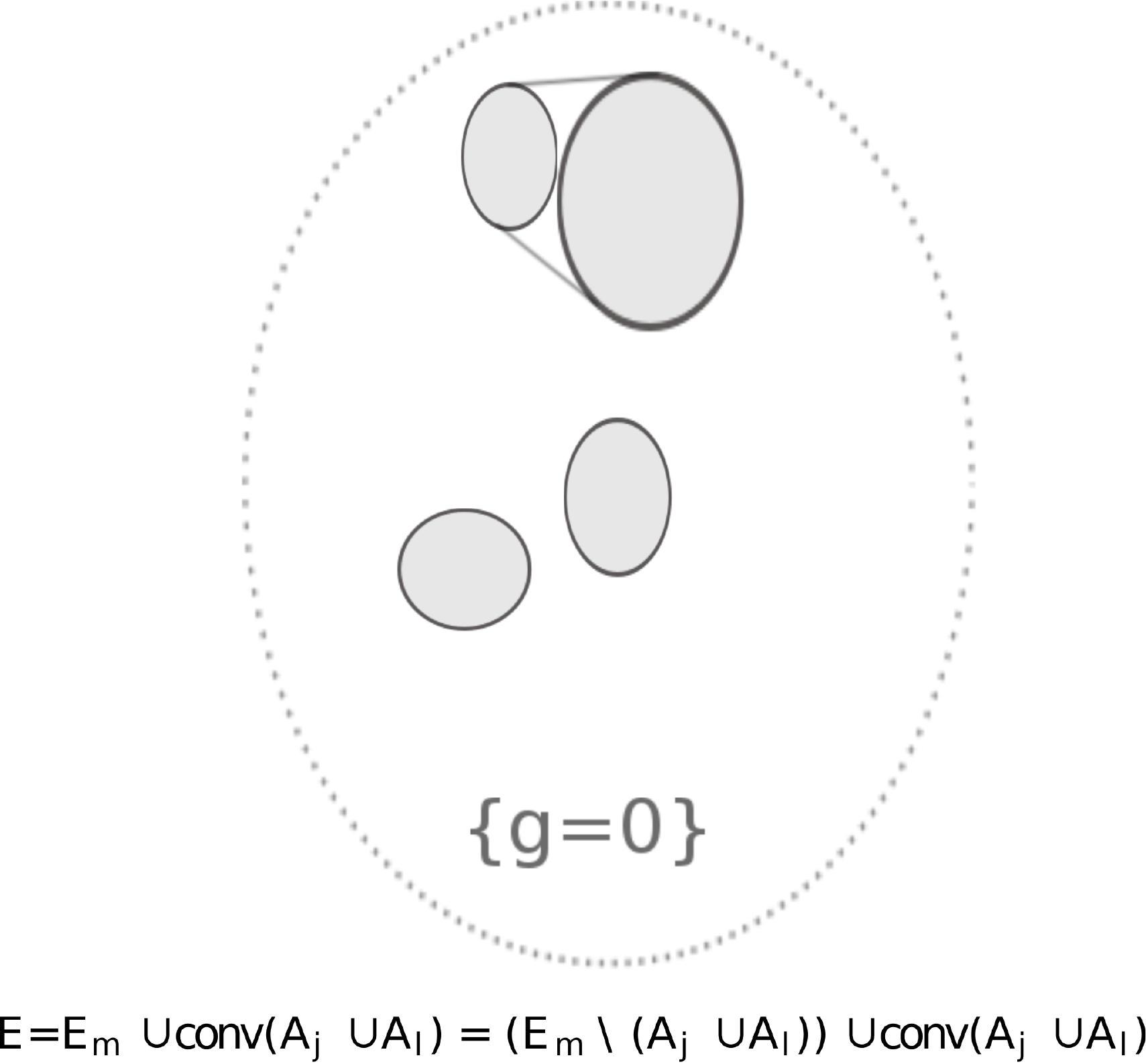}
\caption{(a)}
\end{figure}

 (e.g. via (32) in \cite{MR1641031});
 this implies
\begin{align*}
\mathcal{F}(E) &= \mathcal{F}((E_m \setminus (A_j \cup A _l)) \cup \text{conv}(A_j \cup A _l))\\
&\le \mathcal{F}(E_m \setminus (A_j \cup A _l))+\mathcal{F}(\text{conv}(A_j \cup A _l))\\
& \le \mathcal{F}(E_m \setminus (A_j \cup A _l))+\mathcal{F}(A_j \cup A _l)= \mathcal{F}(E_m)
\end{align*}
which implies
\begin{equation} \label{E}
\mathcal{F}(E) \le \mathcal{F}(E_m).
\end{equation}

If $[E]_m=\{g\le \lambda\} \cap E$, where $\lambda$ satisfies $|[E]_m|=m$ if $E$ is not in $\{g=0\}$, 
$$
\mathcal{F}([E]_m) \le \mathcal{F}(E), 
$$

$$
\mathcal{G}([E]_m) \le \mathcal{G}(E_m)
$$
\\
with strict inequality unless (a) $\mathcal{G}(E_m)=0$; (b) $E_m=[E]_m$
(via e.g. Lemma 2.6  $\&$ Lemma 3.4 in \cite{MR1641031}).
This yields 

$$
\mathcal{E}([E]_m) \le \mathcal{E}(E_m) \le \mathcal{E}([E]_m).
$$
\\
\noindent In particular, one of (a) or (b) must be true: if (a) is true $E_m \subset \{g=0\}$, 
$\mathcal{E}(E_m)=\mathcal{F}(E_m)$, therefore contracting with $a<1$, $|aE|=m$, $E\subset \{g =0\}$ via the convexity of  
$\{g =0\}$ (Figure 3), hence 

$$\mathcal{E}(aE)=a\mathcal{E}(E)<\mathcal{E}(E) \le \mathcal{E}(E_m)$$ 
\\
\noindent because \eqref{E} is true, which contradicts that $E_m$ is a minimizer; suppose (b) is true, 
$$E_m=[E]_m=\{g\le \lambda\} \cap E= \{g\le \lambda\} \cap (E_m \cup (\text{conv}(A_j \cup A _l)))\\$$
\\
\noindent and this implies $E_m \subset \{g\le \lambda\}$; thus, since $\{g\le \lambda\}$ is convex, $E \subset \{g\le \lambda\}$ which implies $E=E_m$ contradicting $|E|>m$.

Hence, $\{A_{i}\}$ have disjoint closures. Now, there exists $w(m)>0$ such that $\inf_i |A_{i}|\ge w(m) > 0$: if $|A_{i}|\rightarrow 0$, let $E_1= A_{1}$, $E_2=A_{i}$,

$$
|hE_1|=|E_1|+|E_2|
$$ 
\\
$|E_2|=(h^2-1)|E_1|=\gamma r^2,$ $r=\sqrt{h^2-1}$, $\gamma>0$. Thus, $\mathcal{F}(E_2) \ge c\sqrt{h^2-1}$, and if $h>1$ is sufficiently near $1$, 

\begin{align*}
&\mathcal{F}(hE_1)+\int_{hE_1} g(x)dx \\ 
&\le \mathcal{F}(E_1) +c_1(h-1)+\int_{E_1} g(x)dx + c_2|hE_1 \setminus E_1|+\mathcal{F}(E_2)+\int_{E_2} g(x)dx\\
&\hskip .3in     -\mathcal{F}(E_2)-\int_{E_2} g(x)dx\\
&\le \mathcal{F}(E_1)+\mathcal{F}(E_2)+\int_{E_1} g(x)dx+\int_{E_2} g(x)dx+\hat c(h-1)-c\sqrt{h^2-1}\\
&<\mathcal{F}(E_1)+\mathcal{F}(E_2)+\int_{E_1} g(x)dx+\int_{E_2} g(x)dx.
\end{align*}
Hence, assuming there exists $0<h-1$ sufficiently small with 

\begin{equation} \label{e}
h E_1 \cap (E_m \setminus E_1) = \emptyset,
\end{equation}
\\
the inequality yields
$$
 \mathcal{E}(hE_1)<\mathcal{E}(E_1 \cup E_2)
$$
and 

$$
|hE_1 \cup (E_m \setminus (E_1 \cup E_2))|=m,
$$
\\

\begin{figure}[htbp]
\includegraphics[width=.8 \textwidth]{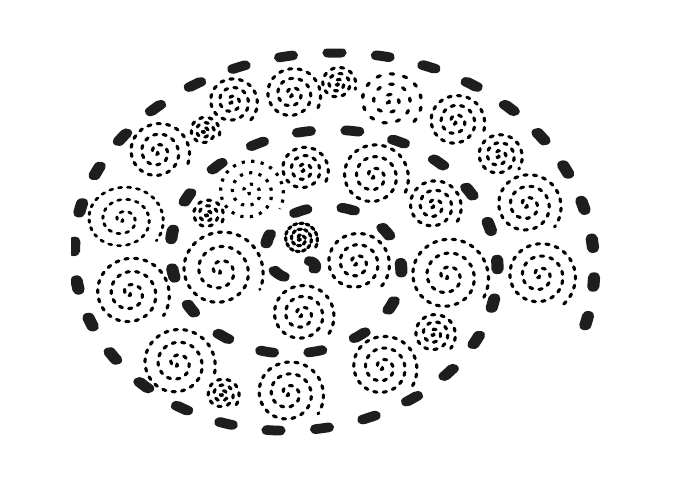}
\caption{A pathological scenario}
\label{dra}
\end{figure}

\noindent a contradiction which finishes the proof. Supposing a pathological scenario in which \eqref{e} is not true: for any two sets in $\{A_i\}$ the h-dilation of one of the sets, say $A_e(=:E_1)$, contains sets $A_k^h \in \{A_i\} \setminus \{A_e, E_2\}$ ($E_2$ is the second set), see Figure \ref{dra}, 
$$
h E_1 \cap (E_m \setminus E_1) \neq \emptyset
$$
\\
for $h\rightarrow 1^+$, 
$$
|h E_1 \cap (E_m \setminus E_1)| \rightarrow 0,
$$

\noindent therefore one may argue as in Proposition A.9 in \cite{MR1641031} to obtain a contradiction.
\end{proof}

\begin{rem}
If $g \in C^{1,\alpha}$ is convex coercive, $f \in C^{3,\alpha}(\mathbb{R}^2\setminus \{0\})$ is uniformly elliptic, connectedness of equilibrium shapes was obtained by De Philippis and Goldman for $m \in (0,\infty)$ \cite{D}. Therefore, their theorem combined with my theorem yields convexity in $\mathbb{R}^2$ subject to the regularity assumptions and without McCann's theorem. Another proof of convexity when $f$ is euclidian is via Theorem 1 in Ferriero and Fusco \cite{MR2461811} $\&$ Proposition 1 in Figalli and Maggi \cite{MR2807136}: 
$$
\mathcal{H}^1(\partial \text{co}(E_m))\le \mathcal{H}^1(\partial E_m) \le \mathcal{H}^1(\partial \text{co}(E_m))
$$
yields
$\text{co}(E_m)=E_m$ \cite{MR2461811}. If $f$ is elliptic, the same argument above and Corollary 2.8 in McCann \cite{MR1641031} (see the proof) replacing Theorem 1 in Ferriero and Fusco \cite{MR2461811} implies convexity.
\end{rem}

\begin{cor}
If $A \subset \mathbb{R}^2$ is bounded, convex and $$\mathcal{F}(E_m)=\inf\{\mathcal{F}(E): E\subset A, |E|=m\},$$ $|K_a|<m\le |A|$, $|K_a|$ is the measure of the largest Wulff shape in $A$,  then $E_m$ is convex. 
\end{cor}

\begin{proof}
Let 
\begin{equation*}  
g(x)=\begin{cases}
 0 & \text{if } x\in \bar A\\
\infty  & \text{if }x \notin \bar A.
\end{cases}
\end{equation*}
If $m \le |A|$, note that a compactness argument implies the existence of minimizers $E_m$. An analog of the proof of Theorem \ref{mi} implies $E_{m}=\cup_{i=1}^N A_{i} \subset \{g=0\}$, $A_i$ convex, disjoint. In particular, the argument to exclude (a) in the proof of the theorem implies $N=1$. 
\end{proof}

\begin{rem}
The problem in the case of the isotropic perimeter was the main objective in \cite{MR1669207} $\&$ has appeared in the articles: \cite[Theorem 11]{MR2178065}, \cite[(1.8)]{MR2436794}, and \cite[Remark 2.6.]{MR2468216}.
\end{rem}

\begin{cor} \label{mil}
Assume $n=2$, the sub-level sets $\{g < t\}$ are convex, $g$ is locally Lipschitz, and  $g$ is coercive, then 
$$E_{m}=\cup_{i=1}^N A_{i},$$
where $A_{i}$ are convex and have disjoint closures, $N<\infty$.
\end{cor}

\begin{rem}
In the case when $g$ is convex and coercive, McCann proved the result in the corollary \cite{MR1641031}. 
\end{rem}

\begin{thm} \label{q}
If $n=2$, the sub-level sets $\{g < t\}$ are convex, $g$ is locally Lipschitz, and $g$ admits minimizers $E_m \subset B_{R(m)}$, then
$\{m: E_m \hskip .03in is \hskip .03in convex \}$
is open. 
\end{thm}

\begin{proof}
Assume $m>0$ and $E_m$ is convex; if $m_k>m$, $E_{m_k}$ are not convex, and $m_k \rightarrow m$, let $R>0$ satisfy $E_{m_k} \subset B_R$,  
$$E_{m_k}=\cup_{i=1}^N A_{k,i},$$

$A_{k,i}$ convex and have disjoint closures (Theorem \ref{mi}); $A_{i}=\liminf_k A_{k,i}$ (in $L^1(B_R)$). 

There exists $w(m)>0$ such that $|A_{k,i}|\ge w(m) > 0$: if $|A_{k,i}|\rightarrow 0$, let $\inf_k |A_{k,l}| \ge w(N,m)>0$ ($N$ is bounded; therefore such a constant exists); then, set $E_1= A_{k,l}$, $E_2=A_{k,i}$,

$$
|hE_1|=|E_1|+|E_2|
$$ 
\\
$|E_2|=(h^2-1)|E_1|=\gamma r^2,$ $r=\sqrt{h^2-1}$, $\gamma>0$. Thus, $\mathcal{F}(E_2) \ge c\sqrt{h^2-1}$, and if $h>1$ is sufficiently near $1$, 

\begin{align*}
&\mathcal{F}(hE_1)+\int_{hE_1} g(x)dx \\ 
&\le \mathcal{F}(E_1) +c_1(h-1)+\int_{E_1} g(x)dx + c_2|hE_1 \setminus E_1|+\mathcal{F}(E_2)+\int_{E_2} g(x)dx\\
&\hskip .3in     -\mathcal{F}(E_2)-\int_{E_2} g(x)dx\\
&\le \mathcal{F}(E_1)+\mathcal{F}(E_2)+\int_{E_1} g(x)dx+\int_{E_2} g(x)dx+\hat c(h-1)-c\sqrt{h^2-1}\\
&<\mathcal{F}(E_1)+\mathcal{F}(E_2)+\int_{E_1} g(x)dx+\int_{E_2} g(x)dx.
\end{align*}
Hence, $$\sum_i |A_{k,i}|=m_k$$

$$
\sum_i\mathcal{F}(A_{k,i})=\mathcal{F}(E_{m_k}),
$$

and $A=\cup_i A_i$ has $|A|=m$;

therefore if $s_k>0$, $|s_kE_m|=|E_{m_k}|$, $s_k \rightarrow 1$, 

$$\mathcal{E}(E_m)\le \mathcal{E}(A)\le \liminf_k \mathcal{E}(E_{m_k}) \le \liminf_k \mathcal{E}(s_kE_m)=\mathcal{E}(E_m)$$
$$
\int_{A_{k_l,i}} g(x)dx \rightarrow \int_{A_i} g(x)dx;
$$
this implies that $A$ is a minimizer with mass $m$ \& 
\begin{align*}
|\mathcal{F}(A) - \sum_i\mathcal{F}(A_{i})|& \le |\mathcal{E}(A)-\mathcal{E}(E_{m_k})|+|\mathcal{G}(E_{m_k})-\mathcal{G}(A)|\\
&  + |\mathcal{F}(E_{m_k}) - \sum_i\mathcal{F}(A_{i})|\\
& \hskip .15in \rightarrow 0,
\end{align*}
which contradicts that minimizers with mass $m$ are convex (since $|A_{i}|\ge w(m) > 0, A=\cup_i A_i$). 
Hence there exists $e_a>0$ such that for all $\bar m\in (m, m+e_a)$, $E_{\bar m}$ is convex. A symmetric argument yields $e_a>0$ such that for all $\bar m\in (m-e_a, m)$, $E_{\bar m}$ is convex.
\end{proof}

\begin{cor} \label{qw}
If $n=2$, the sub-level sets $\{g < t\}$ are convex, $g$ is locally Lipschitz, and $g$ is coercive, then
$\{m: E_m \hskip .03in is \hskip .03in convex \}$
is open. 
\end{cor}

An important application of Corollary \ref{qw} is the enlarging of the small mass convexity in Figalli and Maggi \cite{MR2807136}: \\

``Theorems 1 and 2 deal with the connectedness and convexity properties
of liquid drops and crystals in the small mass regime. Outside this special regime,
one expects convexity of minimizers, provided g is convex"...``this was proved in [8,12] when the mass is large enough. The
natural problem of how to fill the gap in between these two results is open. It seems
very likely that new ideas are needed to deal with this case"
p. 149. \\

In the following, the interval in Theorem 1 in Figalli and Maggi \cite{MR2807136} is extended with merely the assumption of convex sub-level sets and a local Lipschitz regularity assumption (if $g$ is convex, both of these are true). Also, a stability result is proved via uniqueness.

\begin{cor} \label{m+e}
If $n=2$ and $m_c$ is the critical mass in  Figalli and Maggi's Theorem 1 (for all $m \in (0,m_c]$, $E_m$ is convex, \cite{MR2807136}), then if $g$ is locally Lipschitz and the sub-level sets $\{g < t\}$ are convex there exists $e>0$ such that for $m \in (m_c, m_c+e)$, $E_m$ is convex. 
\end{cor}

\begin{cor} \label{m+e2}
If $n=2$, $m_c$ is the critical mass in  Figalli and Maggi's Theorem 1 (for all $m \in (0,m_c]$, $E_m$ is convex, \cite{MR2807136}), and if $g$ is convex there exists $e>0$ such that for $m \in (0, m_c+e)$, $E_m$ is unique $\&$ convex and for $\epsilon>0$ there exists $w_m(\epsilon)>0$ such that if 
$|E|=|E_m|$, $E \subset B_R$, and 
$$
|\mathcal{E}(E)-\mathcal{E}(E_m)|<w_m(\epsilon),
$$
then there exists $x \in \mathbb{R}^2$ such that 
$$
\frac{|(E_m+x) \Delta E|}{|E_m|}<\epsilon.
$$ 
\end{cor}

If the mass is small, the convexity of the sub-level is not necessary to obtain the uniqueness and convexity of minimizers. Therefore, if $g$ is merely in $L_{loc}^\infty$, by enlarging $m$ one expects non-convex minimizers or at least two convex minimizers. The critical mass when this occurs is identified via Theorem \ref{8p} $\&$ in higher dimension with additional regularity in Theorem \ref{85}. If the sub-levels of $g$ are convex $\&$ a uniqueness assumption in the class of convex sets holds, the critical mass is completely identified in $\mathbb{R}^2$ via Theorem \ref{8qp}.

\begin{thm} \label{85}
If $g \in C^1$, $f \in C^{2, \alpha}(\mathbb{R}^n\setminus\{0\})$, $\alpha \in (0,1)$ is $\lambda-$elliptic
and $g$ admits minimizers $E_m \subset B_{R(m)}$ with $R \in L_{loc}^\infty(\mathbb{R}^+)$, either:\\
(i) $E_m$ is convex $\&$ unique for all $m \in (0,\infty)$;\\
(ii) there exists $\mathcal{M}>0$ such that for all $m \in (0, \mathcal{M})$, $E_m$ is unique, convex and there exist $\epsilon_0, \gamma>0$ such that for all $\epsilon \le \epsilon_0$,
$$
\liminf_{m \rightarrow \mathcal{M}^-} \frac{\gamma(\mathcal{M}^{\frac{n-1}{n}}-m^{\frac{n-1}{n}})}{w_m(\epsilon)} \ge 1,
$$
where $w_m(\epsilon)>0$ satisfies Proposition \ref{K};\\
(iii) there exists $\mathcal{M}>0$ such that for all $m \in (0, \mathcal{M}]$, $E_m$ is unique, convex and for $m>\mathcal{M}$ there exists $a<m$ such that either convexity or uniqueness fails for mass $a$.
\end{thm}

\begin{proof}
Define 
$$
\mathcal{A}=\{m: E_m \hskip .08in \text{is unique $\&$ convex}\}
$$
$$
\mathcal{M}=\sup \mathcal{A}.
$$
Theorem \ref{@'} and Theorem 2 in Figalli and Maggi \cite{MR2807136} imply $(0,m_a) \subset \mathcal{A}$, hence $\mathcal{M} >0$. If $\mathcal{M}<\infty$, for $m \in (0,\mathcal{M})$, $E_m$ is unique $\&$ convex. 
Therefore either: (a) there exists a non-convex minimizer having mass $\mathcal{M}$; (b) there exist two convex minimizers having mass $\mathcal{M}$; or (c) for all $m \in (0, \mathcal{M}]$, $E_m$ is unique, convex and for $m>\mathcal{M}$ there exists $a<m$ such that either convexity or uniqueness fails for minimizers with mass $a$.
If $m_k<\mathcal{M}$, $m_k \rightarrow \mathcal{M}$, along a subsequence, $E_{m_k} \rightarrow T_\mathcal{M}$, with $|T_\mathcal{M}|=\mathcal{M}$, $T_\mathcal{M}$ a convex minimizer. Set 
$$
\epsilon=\frac{1}{5} \inf_{x} \frac{|(E_\mathcal{M}+x) \Delta T_\mathcal{M}|}{|E_\mathcal{M}|}>0,
$$
where if (a) $E_\mathcal{M}$ is the non-convex minimizer and if (b) $E_\mathcal{M}$ is a convex minimizer not (mod sets of measure zero and translations) equal to $T_\mathcal{M}$.
If $m \in (0,\mathcal{M})$, the uniqueness of convex minimizers  implies that there exists $w_m(\epsilon)>0$ such that for all $\epsilon>0$, if 
$|E|=|E_m|$, $E \subset B_R$, and 
$$
|\mathcal{E}(E)-\mathcal{E}(E_m)|<w_m(\epsilon),
$$
then there exists $x \in \mathbb{R}^n$ such that 
$$
\frac{|(E_m+x) \Delta E|}{|E_m|}<\epsilon
$$
via Proposition \ref{K}.
Let $\{m_k\}$ be the sequence such that 
$$
\liminf_{m \rightarrow \mathcal{M}^-} \frac{\mathcal{M}^{\frac{n-1}{n}}-m^{\frac{n-1}{n}}}{w_m(\epsilon)}=\lim_{k \rightarrow \infty} \frac{\mathcal{M}^{\frac{n-1}{n}}-m_k^{\frac{n-1}{n}}}{w_{m_k}(\epsilon)},
$$
and define $\gamma_k$ via $|\gamma_k E_\mathcal{M}|=|E_{m_k}|$ (i.e. $\gamma_k=(\frac{m_k}{\mathcal{M}})^{\frac{1}{n}}$); note  

\begin{align*}
|\mathcal{E}(\gamma_k E_\mathcal{M})-\mathcal{E}(E_{m_k})|& \le |\mathcal{E}(\gamma_k E_\mathcal{M})-\mathcal{E}(E_\mathcal{M})|+|\mathcal{E}(T_\mathcal{M})-\mathcal{E}(E_{m_k})|\\
&\le\mathcal{F}(E_\mathcal{M})(1-\gamma_k^{n-1})+ (\sup_{B_R}g) |E_\mathcal{M} \Delta (\gamma_k E_\mathcal{M})|\\
&+|\mathcal{E}(T_\mathcal{M})-\mathcal{E}(E_{m_k})|\\
\end{align*}

\begin{align*}
\mathcal{E}(T_{\mathcal{M}})& \le \mathcal{E}(\frac{1}{\gamma_k}E_{m_k})\\
&=\frac{1}{\gamma_k^{n-1}}\mathcal{F}(E_{m_k})+ \int_{\frac{1}{\gamma_k}E_{m_k}}g(x)dx\\
&\le (\frac{1}{\gamma_k^{n-1}}-1)\mathcal{F}(E_{m_k})+(\sup_{B_R}g) |\frac{1}{\gamma_k}E_{m_k} \Delta E_{m_k}|+\mathcal{E}(E_{m_k})
\end{align*}
and similarly thanks to $|\frac{1}{\gamma_k}E_{m_k} \Delta E_{m_k}| \le a(\frac{1}{\gamma_k}-1)$ (e.g via \cite[Lemma 4]{MR2807136}) this implies  

$$
|\mathcal{E}(T_{\mathcal{M}})-\mathcal{E}(E_{m_k})| \le \alpha_p (\frac{1}{\gamma_k^{n-1}}-1)=\alpha (\mathcal{M}^{\frac{n-1}{n}}-m_k^{\frac{n-1}{n}}),
$$

$m_k \thickapprox \mathcal{M}$.

In particular, 
$$
|\mathcal{E}(\gamma_k E_\mathcal{M})-\mathcal{E}(E_{m_k})| \le \gamma (\mathcal{M}^{\frac{n-1}{n}}-m_k^{\frac{n-1}{n}})
$$
where $\gamma=\gamma(\mathcal{M})$.

Suppose 

$$
\liminf_{m \rightarrow \mathcal{M}^-} \frac{\mathcal{M}^{\frac{n-1}{n}}-m^{\frac{n-1}{n}}}{w_m(\epsilon)}<\frac{1}{\gamma}, 
$$

then for $k$ large

$$
|\mathcal{E}(\gamma_k E_\mathcal{M})-\mathcal{E}(E_{m_k})| \le  \frac{ \gamma (\mathcal{M}^{\frac{n-1}{n}}-m_k^{\frac{n-1}{n}})}{w_{m_k}(\epsilon)} w_{m_k}(\epsilon)<w_{m_k}(\epsilon)
$$

and this implies the existence of $x_k$ such that 

$$
\frac{|(E_{m_k}+x_k) \Delta (\gamma_k E_\mathcal{M})|}{|E_{m_k}|}<\epsilon;
$$

however, if $k$ is large, $\gamma_k \thickapprox 1$, which implies 
\begin{align*}
\frac{|(E_{m_k}+x_k) \Delta (\gamma_k E_\mathcal{M})|}{|E_{m_k}|}& \thickapprox \frac{|(T_\mathcal{M}+x_k) \Delta E_\mathcal{M}|}{|E_\mathcal{M}|}\\
 &\ge \inf_x \frac{|(E_\mathcal{M}+x) \Delta T_\mathcal{M}|}{|E_\mathcal{M}|} =5\epsilon,
\end{align*}
a contradiction.
Therefore
$$
\liminf_{m \rightarrow \mathcal{M}^-} \frac{\mathcal{M}^{\frac{n-1}{n}}-m_k^{\frac{n-1}{n}}}{w_m(\epsilon)}\ge \frac{1}{\gamma}, 
$$

for $$\epsilon \le \epsilon_0:=\frac{1}{5}\inf_x \frac{|(E_\mathcal{M}+x) \Delta T_\mathcal{M}|}{|E_\mathcal{M}|}.$$

\end{proof}

\noindent {\bf Example:} \\
If $g=0$, $a(m,\epsilon)=w_m(\epsilon)=c(n)\epsilon^2 m^{\frac{n-1}{n}}$ via Figalli-Maggi-Pratelli \cite{MR2672283}:  this yields for all $\mathcal{M}, \epsilon, \gamma>0$,
$$
\liminf_{m \rightarrow \mathcal{M}^-} \frac{\gamma(\mathcal{M}^{\frac{n-1}{n}}-m^{\frac{n-1}{n}})}{w_m(\epsilon)} = \liminf_{m \rightarrow \mathcal{M}^-} \frac{\gamma(\mathcal{M}^{\frac{n-1}{n}}-m^{\frac{n-1}{n}})}{c(n)\epsilon^2 m^{\frac{n-1}{n}}}=0
$$
which precludes (ii).

\begin{cor}
If $g \in C^1$ is coercive, $f \in C^{2, \alpha}(\mathbb{R}^n\setminus\{0\})$, $\alpha \in (0,1)$ is $\lambda-$elliptic, either:\\
(i) $E_m$ is convex $\&$ unique for all $m \in (0,\infty)$;\\
(ii) there exists $\mathcal{M}>0$ such that for all $m \in (0, \mathcal{M})$, $E_m$ is unique, convex and there exist $\epsilon_0, \gamma>0$ such that for all $\epsilon \le \epsilon_0$,
$$
\liminf_{m \rightarrow \mathcal{M}^-} \frac{\gamma(\mathcal{M}^{\frac{n-1}{n}}-m^{\frac{n-1}{n}})}{w_m(\epsilon)} \ge 1,
$$
where $w_m(\epsilon)>0$ satisfies Proposition \ref{K};\\
(iii) there exists $\mathcal{M}>0$ such that for all $m \in (0, \mathcal{M}]$, $E_m$ is unique, convex and for $m>\mathcal{M}$ there exists $a<m$ such that either convexity or uniqueness fails for mass $a$.
\end{cor}

\begin{thm} \label{8p}
If $g \in L_{loc}^\infty$, $n=2$, and $g$ admits minimizers $E_m \subset B_{R(m)}$ with $R \in L_{loc}^\infty(\mathbb{R}^+)$, either:\\
(i) $E_m$ is convex $\&$ unique for all $m \in (0,\infty)$;\\
(ii) there exists $\mathcal{M}>0$ such that for all $m \in (0, \mathcal{M})$, $E_m$ is unique, convex and there exist $\epsilon_0, \gamma>0$ such that for all $\epsilon \le \epsilon_0$,
$$
\liminf_{m \rightarrow \mathcal{M}^-} \frac{\gamma(\mathcal{M}-m)}{w_m(\epsilon)} \ge 1,
$$
where $w_m(\epsilon)>0$ satisfies Proposition \ref{K};\\
(iii) there exists $\mathcal{M}>0$ such that for all $m \in (0, \mathcal{M}]$, $E_m$ is unique, convex and for $m>\mathcal{M}$ there exists $a<m$ such that either convexity or uniqueness fails for mass $a$.
\end{thm}

\begin{proof}
Define 
$$
\mathcal{A}=\{m: E_m \hskip .08in \text{is unique $\&$ convex}\}
$$
$$
\mathcal{M}=\sup \mathcal{A}.
$$
Theorem \ref{@'} and Theorem 1 in Figalli and Maggi \cite{MR2807136} imply $(0,m_a) \subset \mathcal{A}$, hence $\mathcal{M} >0$. Thus one may  argue -- verbatim -- as in the proof of Theorem \ref{85}.
\end{proof}

\begin{cor}
If $g$ is coercive, locally Lipschitz, $n=2$, either:\\
(i) $E_m$ is convex $\&$ unique for all $m \in (0,\infty)$;\\
(ii) there exists $\mathcal{M}>0$ such that for all $m \in (0, \mathcal{M})$, $E_m$ is unique, convex and there exist $\epsilon_0, \gamma>0$ such that for all $\epsilon \le \epsilon_0$,
$$
\liminf_{m \rightarrow \mathcal{M}^-} \frac{\gamma(\mathcal{M}-m)}{w_m(\epsilon)} \ge 1,
$$
where $w_m(\epsilon)>0$ satisfies Proposition \ref{K};\\
(iii) there exists $\mathcal{M}>0$ such that for all $m \in (0, \mathcal{M}]$, $E_m$ is unique, convex and for $m>\mathcal{M}$ there exists $a<m$ such that either convexity or uniqueness fails for mass $a$.
\end{cor}

\begin{thm} \label{8qp}
If $n=2$, the sub-level sets $\{g < t\}$ are convex, $g$ is locally Lipschitz, and $g$ admits minimizers $E_m \subset B_{R(m)}$ with $R \in L_{loc}^\infty(\mathbb{R}^+)$ which if convex are unique within the class of convex sets,  either:\\
(i) $E_m$ is convex for all $m \in (0,\infty)$;\\
(ii) there exists $\mathcal{M}>0$ such that for all $m \in (0, \mathcal{M})$, $E_m$ is convex and there exist $\epsilon_0, \gamma>0$ such that for all $\epsilon \le \epsilon_0$,
$$
\liminf_{m \rightarrow \mathcal{M}^-} \frac{\gamma(\mathcal{M}-m)}{w_m(\epsilon)} \ge 1,
$$
where $w_m(\epsilon)>0$ satisfies Proposition \ref{K}.
\end{thm}

\begin{proof}
Define 
$$
\mathcal{A}=\{m: E_m \hskip .08in \text{is convex}\}
$$
$$
\mathcal{M}=\sup \mathcal{A}.
$$
Corollary \ref{m+e} implies $(0,m_a) \subset \mathcal{A}$, hence $\mathcal{M} >0$. Assume $E_{\mathcal{M}}$ is convex, then Theorem \ref{q} implies the existence of $\epsilon>0$ such that $E_m$ is convex for all $m \in (\mathcal{M}, \mathcal{M}+\epsilon)$ contradicting $\mathcal{M}=\sup \mathcal{A}$. Thus, for $m<\mathcal{M}$, $E_m$ is convex and there exists a non-convex minimizer $E_\mathcal{M}$. If $m_k<\mathcal{M}$, $m_k \rightarrow \mathcal{M}$, along a subsequence, $E_{m_k} \rightarrow T_\mathcal{M}$, with $|T_\mathcal{M}|=\mathcal{M}$, $T_\mathcal{M}$ a convex minimizer. Set 
$$
\epsilon=\frac{1}{5} \inf_{x} \frac{|(E_\mathcal{M}+x) \Delta T_\mathcal{M}|}{|E_\mathcal{M}|}>0
$$
and observe that the argument in Theorem \ref{85}'s proof yields (ii).
\end{proof}

\begin{rem}
Suppose $g$ is convex $\&$ coercive, then the assumptions in the theorem hold.
\end{rem}

\noindent {\bf Example:} \\
If $g=0$, $n=2$, $a(m,\epsilon)=w_m(\epsilon)=c\epsilon^2 m^{\frac{1}{2}}$ via Figalli-Maggi-Pratelli \cite{MR2672283}:  this yields for all $\mathcal{M}, \epsilon, \gamma>0$,
$$
\liminf_{m \rightarrow \mathcal{M}^-} \frac{\gamma(\mathcal{M}-m)}{w_m(\epsilon)} = \liminf_{m \rightarrow \mathcal{M}^-} \frac{\gamma(\mathcal{M}-m)}{c\epsilon^2 m^{\frac{1}{2}}}=0
$$
which precludes (ii).

\subsection{Proof of Theorem \ref{n=2}}

\begin{proof}
If $E_m=\cup_{k=1}^N A_k$, $\{A_k\}$ convex, disjoint, non-empty, $N>1$, there exists $\bar K$ such that
$$
|A_{\bar K} \cap \{g \neq 0\}|>0,
$$  
$0 \notin A_{\bar K}$: assume not, let $E=E_m \cup \text{conv}(A_j \cup A _l) \subset \{g=0\}$ via the convexity of  
$\{g =0\}$, see Figure 3; firstly observe that $|E|>m$, and since by a translation, the closure of $A_j \cup A _l$ is connected, 
$$
\mathcal{F}(\text{conv}(A_j \cup A _l)) \le \mathcal{F}(A_j \cup A _l)
$$
 (via e.g. Corollary 2.8 in \cite{MR1641031}); this implies
\begin{align*}
\mathcal{F}(E) &\le \mathcal{F}(E_m \setminus (A_j \cup A _l))+\mathcal{F}(\text{conv}(A_j \cup A _l))\\
& \le \mathcal{F}(E_m \setminus (A_j \cup A _l))+\mathcal{F}(A_j \cup A _l)= \mathcal{F}(E_m)
\end{align*}
which implies

\begin{equation} \label{E@}
\mathcal{F}(E) \le \mathcal{F}(E_m)
\end{equation}
therefore contracting with $a<1$, $|aE|=m$, $E\subset \{g =0\}$, hence 
$$\mathcal{E}(aE)=a\mathcal{E}(E)<\mathcal{E}(E) \le \mathcal{E}(E_m)$$ 
because \eqref{E@} is true, which contradicts that $E_m$ is a minimizer; 

hence there exists $\bar K$ such that
$$
|A_{\bar K} \cap \{g \neq 0\}|>0,
$$  
$0 \notin A_{\bar K}$; (iv) implies 
$$
\int_{A_{\bar K}} \nabla g(x) dx \neq 0.
$$
If $\nu \in \mathbb{S}^{1}$ and $t>0$ is small,

$$
\int_{A_{\bar K}+t\nu} g(x) dx \ge \int_{A_{\bar K}} g(x) dx
$$
since

$$
\int_{A_{\bar K}+t\nu} g(x) dx < \int_{A_{\bar K}} g(x) dx \Rightarrow \mathcal{E}(E_m) > \mathcal{E}((\cup_{k \neq \bar K} A_k) \cup  A_{\bar K+t\nu}),
$$
where for sufficiently small $t>0$, $|(\cup_{k \neq \bar K} A_k) \cup  A_{\bar K+t\nu}|=m$, and hence this contradicts $E_m$ having the smallest energy;
$$
\Rightarrow \int_{A_{\bar K}} g(x+t\nu)-g(x) dx \ge 0
$$
for $\nu \in \mathbb{S}^1$, $t>0$ small. Fix $\nu \in \mathbb{S}^1$, then for $t>0$ small
(i) yields
$$
\frac{|g(x+t\nu)-g(x)|}{t} \le M
$$
and dominated convergence implies

$$
\int_{A_{\bar K}} \nabla g(x) \cdot \nu dx \ge 0;
$$
$$
\Rightarrow \nu \cdot \int_{A_{\bar K}} \nabla g(x)  dx \ge 0
$$
for $\nu \in \mathbb{S}^1$; since 

$$
\int_{A_{\bar K}} \nabla g(x)  dx \neq 0
$$

$$
\Rightarrow \Big(\frac{-\int_{A_{\bar K}} \nabla g(x) }{|\int_{A_{\bar K}} \nabla g(x)|}\Big) \cdot \int_{A_{\bar K}} \nabla g(x)  dx \ge 0
$$

$$
\Rightarrow \int_{A_{\bar K}} \nabla g(x)  dx = 0,
$$
and this contradicts (iv).
\end{proof}

\begin{rem} 
The result is the first general convexity theorem for $m \in (0,|\{g<\infty\}|)$ with the convexity of sub-level sets $\{g<t\}$ (instead of a convexity assumption on $g$) \& includes non-convex functions. Moreover, the 
condition:\\
if $E\subset \{g<\infty\}$ is bounded convex, $0 \notin E$, $|E \cap \{g \neq 0\}|>0$, then 
$$
\int_E \nabla g(x) dx \neq 0,
$$
encodes the gravitational potential
\begin{equation*}  
g(x)=\begin{cases}
\alpha x_n & \text{if } x_n\ge 0\\
\infty  & \text{if }x_n<0
\end{cases}
\end{equation*}
because 
$$
 \int_E \nabla g(x) dx=(0,\alpha |E|)\neq 0,
$$
therefore this generates a unified theory.
\end{rem}

\begin{cor}
If $n=2$, $\phi(0)= 0$, $\phi'>0$, and\\ 
(i) 
\begin{equation*}  
g(x)=\begin{cases}
 \phi(x_n) & \text{if } x_n\ge 0\\
\infty  & \text{if }x_n<0
\end{cases}
\end{equation*}
(ii) $\mathcal{F}$ satisfies assumptions for the existence of minimizers $E_m \subset B_{R(m)}$ with $R \in L_{loc}^\infty(\mathbb{R}^+)$, \\

\noindent then $E_m$ is convex for all $m \in (0,\infty)$.
\end{cor}

\begin{rem}
The simple case is when $\mathcal{F}=\mathcal{H}^1$: existence follows via Steiner symmetrization \cite{MR2178968, MR3055761} $\&$ compactness. Likewise, one may obtain existence for $\mathcal{F}$ which have symmetric tensions with respect to $\{x_1=0\}$, e.g. $f$ is admissible \cite{R}; see also \cite{MR2245755}. Classically, $\phi(x_n)=\alpha x_n$ generates the gravitational potential. The above generalization is new.
\end{rem}

\begin{cor}
If $n=2$ and\\ 
(i) g is locally Lipschitz in $\{g<\infty\}$\\
(ii) $g$ is coercive \\ 
(iii) the sub-level sets $\{g<t\}$ are convex\\
(iv) when $E \subset \{g<\infty\}$ is bounded convex, $0 \notin E$, $|E \cap \{g \neq 0\}|>0$, then 
$$
\int_E \nabla g(x) dx \neq 0,
$$
then $E_m$ is convex for all $m \in (0,|\{g<\infty\}|)$.
\end{cor}

\begin{cor}
If $n=2$ and\\ 
(i) g is locally Lipschitz\\
(ii) g is coercive \\ 
(iii) the sub-level sets $\{g<t\}$ are convex\\
(iv) when $E$ is bounded convex, $0 \notin E$, $|E \cap \{g \neq 0\}|>0$, then 
$$
\int_E \nabla g(x) dx \neq 0,
$$
then $E_m$ is convex for all $m \in (0,\infty)$.
\end{cor}

\begin{cor} \label{a_radi}
If $n=2$ and\\ 
(i) $g=g(|x|)$ is locally Lipschitz \\
(ii) g admits minimizers $E_m \subset B_{R(m)}$ with $R \in L_{loc}^\infty(\mathbb{R}^+)$ \\ 
(iii) when $g>0$, $g$ is (strictly) increasing \\
then $E_m$ is convex for all $m \in (0,\infty)$.
\end{cor}

\begin{figure} 
{\centering\includegraphics[width=.6 \textwidth]{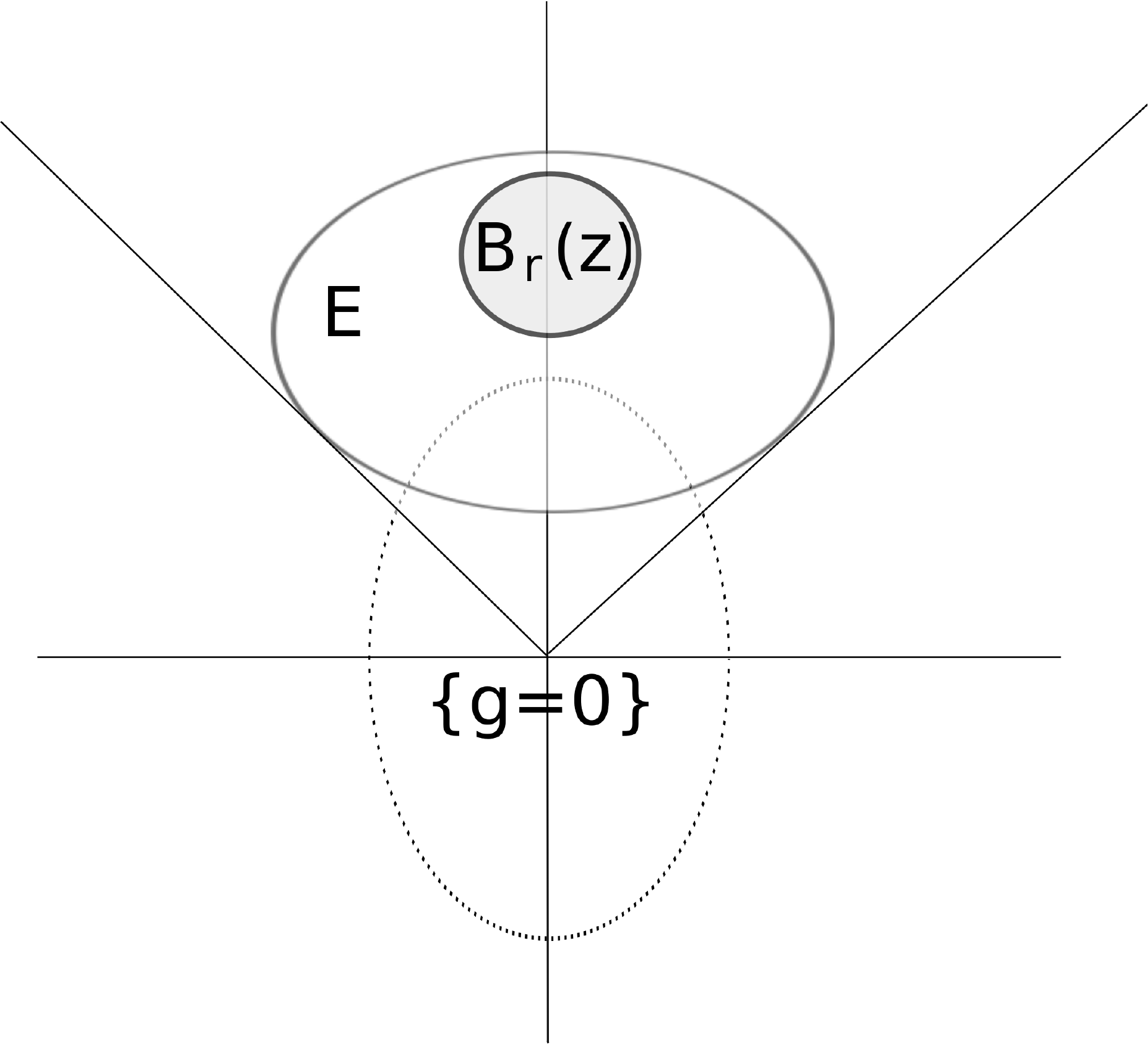}} \caption{$h'(|x|)=0$ on $B_r(z) \subset E \cap \{g \neq 0\}$, Corollary \ref{a_radi} \label{drawingq} }
\end{figure}

\begin{proof}
(i) yields $g(x)=h(|x|)$, $h:\mathbb{R}^+ \rightarrow \mathbb{R}^+$
$$
\Rightarrow \nabla g(x)=h'(|x|)\frac{x}{|x|},
$$
\& if $0 \notin E$ is convex such that $|E \cap \{ g \neq 0\}|>0$

$$
\Rightarrow \int_E \nabla g(x)dx = \int_E h'(|x|) \frac{x}{|x|} dx.
$$
Hence, up to a rotation,

$$
E \subset \text{a convex cone in $\{x_2>0\}$}
$$
thus
$$
\int_E \nabla g(x)dx=0 \Rightarrow \int_E h'(|x|) \frac{x_2}{|x|} dx=0
$$
and since $h' \ge 0$ a.e.
$\Rightarrow h'(|x|)=0$ for a.e. $x \in E$, see Figure \ref{drawingq}; therefore there exists $r>0$ such that  $h'(|x|)=0$ on $B_r(z) \subset E \cap \{g \neq 0\}$

$\Rightarrow g$ is constant on some interval in $\{g > 0\}$, a contradiction to (iii). This yields
$$
\int_E \nabla g(x)dx\neq 0. 
$$

\end{proof}

\begin{cor}
If $n=2$ and\\ 
(i) $g=g(|x|)$ is locally Lipschitz \\
(ii) $g$ is coercive \\ 
(iii) when $g>0$, $g$ is (strictly) increasing \\
then $E_m$ is convex for all $m \in (0,\infty)$.
\end{cor}

\begin{cor}[The radial convex potential] \label{ra}
If $n=2$ and $g=g(|x|)$ is coercive and convex, then $E_m$ is convex for all $m \in (0,\infty)$.
\end{cor}

\begin{rem}
This corollary is new! Assuming in addition $f(v)=f(-v)$, the convexity was obtained by McCann:  Corollary 1.4 in \cite{MR1641031}.
\end{rem}
\begin{rem}
 $g=g(|x|) \ge 0$ coercive and convex is equivalent to $g=g(|x|)\ge 0$ not identically zero and convex.
\end{rem}

\noindent {\bf Acknowledgement:} I want to thank several individuals for their comments on a preliminary version of the paper: Alessio Figalli, Robert McCann, Shibing Chen, Frank Morgan, Laszlo Lempert, and John Andersson. The interactions added quality. I want to especially thank Alessio, Robert, and Shibing for their energy and time investment.

\section{Appendix}

\subsection{Modulus of the free energy}
\begin{prop} \label{K}
If $m>0$, $g \in L_{loc}^\infty$, and up to translations and sets of measure zero, $g$ admits unique minimizers $E_m$, then for $\epsilon>0$ there exists $w_m(\epsilon)>0$ such that if 
$|E|=|E_m|$, $E \subset B_R$, and 
$$
\mathcal{A}(E,E_m):=|\mathcal{F}(E)-\mathcal{F}(E_m)|+|\int_E g(x)dx-\int_{E_m} g(x)dx|<w_m(\epsilon),
$$
then there exists $x \in \mathbb{R}^n$ such that 
$$
\frac{|(E_m+x) \Delta E|}{|E_m|}<\epsilon.
$$
\end{prop}

\begin{proof}
Assume not, then there exists $\epsilon>0$ and for $w>0$, there exist $E_w'$ and $E_m$, $|E_w'|=|E_m|=m$, 

$$
\mathcal{A}(E_w',E_m)<w
$$

$$
\inf_{x} \frac{|(E_m+x)\Delta E_w'|}{|E_m|} \ge \epsilon;
$$

in particular, set $w=\frac{1}{k}$, $k \in \mathbb{N}$; observe that there exist $E_{\frac{1}{k}}'$, $|E_{\frac{1}{k}}'|=m$,

 $$
|\mathcal{E}(E_m)-\mathcal{E}(E_{\frac{1}{k}}')|\le \mathcal{A}(E_{\frac{1}{k}}',E_m)<\frac{1}{k},
$$

$$
\inf_{x} \frac{|(E_m+x)\Delta E_{\frac{1}{k}}'|}{|E_m|} \ge \epsilon;
$$

thus

$$
\mathcal{F}(E_{\frac{1}{k}}') \le \mathcal{E}(E_{\frac{1}{k}}') < \frac{1}{k}+\mathcal{E}(E_m), 
$$

$E_{\frac{1}{k}}' \subset B_R$,

and the compactness for sets of finite perimeter implies --up to a subsequence--
$$
E_{\frac{1}{k}}' \rightarrow E' \hskip .3in in \hskip .08in L^1(B_R), 
$$
and therefore $|E'|=m$,

$$
\mathcal{E}(E') \le \liminf_k \mathcal{E}(E_{\frac{1}{k}}') =\mathcal{E}(E_m),
$$
which implies
$E'$ is a minimizer contradicting
$$
\inf_{x} \frac{|(E_m+x)\Delta E'|}{|E_m|} \ge \epsilon
$$
via the uniqueness of minimizers.
\end{proof}

\begin{cor}  \label{K*}
If $m>0$, $g \in L_{loc}^\infty$, and up to translations and sets of measure zero, $g$ admits unique minimizers $E_m$,
for $\epsilon>0$ there exists $w_m(\epsilon)>0$ such that if 
$|E|=|E_m|$, $E \subset B_R$, and 
$$
|\mathcal{E}(E)-\mathcal{E}(E_m)|<w_m(\epsilon),
$$
then there exists $x \in \mathbb{R}^n$ such that 
$$
\frac{|(E_m+x) \Delta E|}{|E_m|}<\epsilon.
$$
\end{cor}

\begin{proof}
$$
|\mathcal{E}(E)-\mathcal{E}(E_m)| \le \mathcal{A}(E,E_m).
$$
\end{proof}

\subsection{Convexity in higher dimension}
One feature of Theorem \ref{@'} is the lower bound on the modulus. The result also addresses a conjecture stated in Figalli and Maggi \cite{MR2807136}.\\

{\bf Conjecture: convexity of minimizers in the small mass regime} \\

In the small mass regime, minimizers are connected and uniformly close to a (properly
rescaled and translated) Wulff shape, in terms of the smallness of the mass. The convexity
of these minimizers remains conjectural, with the exception of the planar case $n=2$ and of
the $\lambda$-elliptic case\\
(Figalli and Maggi \cite{MR2807136}, p. 147).\\

Assuming $g \in C^1$, $f \in C^{2, \alpha}(\mathbb{R}^n\setminus\{0\})$, $\alpha \in (0,1)$ is $\lambda-$elliptic, Figalli and Maggi \cite{MR2807136} proved the existence of $m_0=m_0(n,g,f)>0$ such that if $m \le m_0$, $E_m$ is convex [Theorem 2, \cite{MR2807136}]. 

Assume $g \in L_{loc}^\infty$, $m_0(n,g,f)$ is called {\bf stable} if there exist $g_a \rightarrow g$, $f_a \rightarrow f$ pointwise, with $g_a \in C^1$,  $f_a \in C^{2, \alpha}(\mathbb{R}^n\setminus\{0\})$  $\lambda-$elliptic, $\alpha \in (0,1)$ such that: $$\liminf_a m_0(n,g_a,f_a) >0.$$

Note that if $g \in C^1$,  $f \in C^{2, \alpha}(\mathbb{R}^n\setminus\{0\})$  is $\lambda-$elliptic, $\alpha \in (0,1)$, then $m_0(n,g,f)$ is stable because there exist $f_a=f \rightarrow f$, $g_a=g \rightarrow g$ $\&$
$$\liminf_a m_0(n,g_a,f_a)=m_0(n,g,f) >0.$$

\begin{cor}[Corollary of Theorem \ref{@'}] \label{m_s}
If $g \in L_{loc}^\infty$ and $m_0(n,g,f)$ is stable, then the conjecture is true.
\end{cor}

\begin{proof}
Theorem 2 in Figalli and Maggi \cite{MR2807136} implies that for smooth $g_a$, and elliptic $f_a$, if $m$ is sufficiently small, there exists a convex set $E_{m,a}$ which is a minimizer of the free energy with respect to $f_a$, $g_a$ with mass 
$m$. Assume $f$ is not elliptic but $1$-homogeneous and convex. Then there exists $f_a$ elliptic such that $f_a(x) \rightarrow f(x)$ for all $x$; moreover, there exists $\{g_a\}$ smooth with $g_a \rightarrow g$. Thus supposing $m_0(n,g,f)$ is stable, along a subsequence, 
$$
E_{m,a_k} \rightarrow E,
$$    
where $E$ is a convex minimizer of the free energy with respect to $f$, $g$ and $|E|=m < \liminf_a m_0(n,g_a,f_a)$. The above theorem implies that for $m$ sufficiently small, $E$ is the unique minimizer (mod translations and sets of measure zero).  
\end{proof} 

\begin{rem}
Assuming $m_0$ is stable, the theorem in Figalli and Maggi implies the existence of convex minimizers for $m \le m_0$, nevertheless it does not preclude the existence of other minimizers. Theorem \ref{@'} yields all minimizers are convex when $m$ is small. 
\end{rem}

\begin{rem}
If $f$ is crystalline, Figalli and Zhang recently proved the conjecture: for sufficiently small mass minimizers are polyhedra \cite{pFZ}.
\end{rem}

\newpage
\bibliographystyle{amsalpha}
\bibliography{References}

\end{document}